\documentclass[12pt]{amsart}
 
\pdfoutput=1
\usepackage[margin=1in]{geometry}
\usepackage{amsmath,amsthm,amssymb,amsrefs,bbm,color,esint,esvect,float,graphicx,mathrsfs}

\usepackage{euscript,latexsym}
\usepackage{accents}

\usepackage{epstopdf}
\AppendGraphicsExtensions{.tif}

 \newtheorem{thm}{Theorem}[section]
 \newtheorem{lemma}[thm]{Lemma}
 \newtheorem{cor}[thm]{Corollary}
 \newtheorem{prop}[thm]{Proposition}

 \numberwithin{equation}{section}

\theoremstyle{definition}

\begin{document}

\title[Weighted endpoint bounds for Bergman and Cauchy-Szeg\H o projections]{Weighted endpoint bounds for the Bergman and Cauchy-Szeg\H o projections on domains with near minimal smoothness}
\author{Cody B. Stockdale}
\address{Cody B. Stockdale\hfill\break\indent 
 Department of Mathematics and Statistics\hfill\break\indent 
 Washington University in St. Louis\hfill\break\indent 
 One Brookings Drive\hfill\break\indent 
 St. Louis, MO, 63130 USA}
\email{codystockdale@wustl.edu}
\author{Nathan A. Wagner}
\address{Nathan A. Wagner \hfill\break\indent 
 Department of Mathematics and Statistics\hfill\break\indent 
 Washington University in St. Louis\hfill\break\indent 
 One Brookings Drive\hfill\break\indent 
 St. Louis, MO, 63130 USA}
\email{nathanawagner@wustl.edu}
\thanks{N. A. Wagner's research is supported in part by National Science Foundation grant DGE \#1745038.}

\begin{abstract}
We study the Bergman projection, $\mathcal{B}$, and the Cauchy-Szeg\H o projection, $\mathcal{S}$, on bounded domains with near minimal smoothness. We prove that $\mathcal{B}$ has the weak-type $(1,1)$ property with respect to weighted measures assuming that the underlying domain is strongly pseudoconvex with $C^4$ boundary and the weight satisfies the $B_1$ condition, and the same property for $\mathcal{S}$ on domains with $C^3$ boundaries and weights satisfying the $A_1$ condition. We also obtain weighted Kolmogorov and weighted Zygmund inequalities for $\mathcal{B}$ and $\mathcal{S}$ in their respective settings as corollaries.
\end{abstract}

\maketitle

\section{Introduction}
Let $D \subseteq \mathbb{C}^n$ be a domain. The Bergman space, $A^2(D)$, is defined by 
$$
    A^2(D):=L^2(D) \cap \text{Hol}(D),
$$
where $L^2(D)$ is the set of square integrable functions on $D$ (with respect to Lebesgue measure) and $\text{Hol}(D)$ is the set of holomorphic functions on $D$. Since $A^2(D)$ is a closed subspace of $L^2(D)$, there exists an orthogonal projection from $L^2(D)$ to $A^2(D)$ which is called the Bergman projection and denoted by $\mathcal{B}$. The Bergman projection can be viewed as an integral operator by
$$
    \mathcal{B}f(z)= \int_D K(z,w)f(w)\,dV(w),
$$
where $K$ is the reproducing kernel for $A^2(D)$ and $V$ represents the Lebesgue measure on $\mathbb{C}^n=\mathbb{R}^{2n}$. We are interested in investigating the boundedness properties of $\mathcal{B}$. 

By definition, it is clear that $\mathcal{B}$ acts as a bounded operator on $L^2(D)$. However, the $L^p(D)$ boundedness of $\mathcal{B}$ for $p\neq 2$ is more complicated. It was first shown by Zaharjuta and Judovi\v c in \cite{ZJ1964} that $\mathcal{B}$ has a bounded extension on $L^p(\mathbb{D})$ for all $1<p<\infty$, where $\mathbb{D}$ is the unit disk in $\mathbb{C}$. In \cite{FR1975}, Forelli and Rudin proved the $L^p(\mathbb{B}_n)$ estimates, where $\mathbb{B}_n$ is the unit ball in $\mathbb{C}^n$, using Schur's test. Later, in \cite{PS1977}, Phong and Stein generalized this result to strongly pseudoconvex domains $D$ with smooth boundary. More recently, in \cite{LS2012}, Lanzani and Stein relaxed the smoothness condition on $D$ and proved the $L^p(D)$ bounds when $D\subseteq \mathbb{C}^n$ is strongly pseudoconvex with $C^2$ boundary. Their methods rely on techniques from complex analysis, operator theory, and harmonic analysis. 

Notice that $\mathcal{B}$ is not bounded on $L^1(D)$ in general. This fact can be seen by taking $D=\mathbb{B}_n$ and noting that boundedness on $L^1(\mathbb B_n)$ would imply that $\mathcal B$ is bounded on $L^\infty(\mathbb B_n)$ by duality. This would then contradict the well-known Rudin-Forelli estimate
$$ 
    \sup_{z \in \mathbb B_n} \int_{\mathbb{B}_n} \frac{1}{|1-\langle z,w \rangle|^{n+1}}\,dV(w) = \infty,
$$
see \cite{Zhu}.

In certain settings, the failed $L^1(D)$ bound may be replaced with the following weak-type $(1,1)$ property: there exists $C>0$ such that 
$$
\|\mathcal{B}f\|_{L^{1,\infty}(D)}:=\sup_{\lambda>0}\lambda V(\{z\in D : |\mathcal{B}f(z)|>\lambda\}) \leq C\|f\|_{L^1(D)}
$$
for all $f \in L^1(D)$. In other words, $\mathcal{B}$ acts as a bounded operator from $L^1(D)$ to $L^{1,\infty}(D)$, where $L^{1,\infty}(D)$ is the space of functions $f$ on $D$ for which $\|f\|_{L^{1,\infty}(D)}<\infty$. 
In particular, McNeal proved the weak-type $(1,1)$ estimate for three different classes of domains, all with smooth boundary, in \cite{M1994}: finite type domains in $\mathbb{C}^2$, decoupled, finite type domains in $\mathbb{C}^n$, and convex, finite type domains in $\mathbb{C}^n$. The same methods with little modification also apply to strongly pseudoconvex domains with smooth boundary. The proof involves defining an appropriate quasi-metric on $D$ and using real variable techniques. See also \cite{DHZZ2001} for a direct proof of the weak-type $(1,1)$ estimate in the case $D=\mathbb{D}\subseteq \mathbb{C}$. 
 
Recent attention has been given to understanding weighted bounds for the Bergman projection. We call a locally integrable function that is positive almost everywhere a weight. For a weight $\sigma$, we denote
$$
    L_{\sigma}^p(D) := \left\{f:D\rightarrow \mathbb{C} \,\,:\,\, \int_D |f|^p\sigma \, dV < \infty\right\}.
$$
For the case $D=\mathbb{B}_n$, Bekoll\'e proved in \cite{B198182} that $\mathcal{B}$ extends boundedly on $L_{\sigma}^p(\mathbb{B}_n)$ for $1<p<\infty$ if and only if $\sigma$ satisfies the $B_p$ condition: 
$$
[\sigma]_{B_p}:=\sup_{\substack{B(z,r)\\ r> d(z,bD)}}\left(\frac{1}{V(B(z,r))}\int_{B(z,r)}\sigma \, dV\right)\left(\frac{1}{V(B(z,r))}\int_{B(z,r)}\sigma^{-\frac{1}{p-1}} \, dV\right)^{p-1}<\infty.
$$
Here we use the notation $B(z,r)$ to represent a quasi-ball centered at $z$ with radius $r$ and $d(z,bD)$ to denote the quasi-distance from the point $z$ to the boundary of $D$ with respect to a certain quasi-metric defined on $\overline{\mathbb{B}_n}\times\overline{\mathbb{B}_n}$. Bekoll\'e also addressed the case $p=1$ for $D=\mathbb{B}_n$ by proving that $\mathcal{B}$ extends boundedly from $L_{\sigma}^1(\mathbb{B}_n)$ to $L_{\sigma}^{1,\infty}(\mathbb{B}_n)$ if and only if $\sigma$ satisfies the $B_1$ condition: 
$$
[\sigma]_{B_1}:= \sup_{\substack{B(z,r)\\ r> d(z,bD)}}\left(\frac{1}{V(B(z,r))}\int_{B(z,r)}\sigma \, dV\right)\|\sigma^{-1}\|_{L^{\infty}(B(z,r))}<\infty.
$$
The above result for $1<p<\infty$ was recently extended to the near minimal smoothness case where $D$ is a strongly pseudoconvex bounded domain with $C^4$ boundary by the second author and Wick in \cite{WW2020}. In Section \ref{BergmanSection}, we use the same condition for $B_1$ weights with respect to the quasi-metric defined therein.

The first main result of this paper is the weighted weak-type $(1,1)$ estimate for the Bergman projection on domains with near minimal smoothness.
\begin{thm}\label{BergmanWeakType} 
If $D \subseteq \mathbb C^n$ is a strongly pseudoconvex bounded domain with $C^4$ boundary and $\sigma$ is a $B_1$ weight on $D$, then the Bergman projection $\mathcal{B}$ extends boundedly from $L_{\sigma}^1(D)$ to $L_{\sigma}^{1,\infty}(D)$. That is, there exists $C>0$ such that 
$$
\|\mathcal{B}f\|_{L_{\sigma}^{1,\infty}(D)}:=\sup_{\lambda>0}\lambda\sigma(\{z \in D: |\mathcal{B}f(z)|>\lambda\})\leq C\|f\|_{L_{\sigma}^1(D)}
$$
for all $f \in L_{\sigma}^1(D)$.
\end{thm}

We remark that Theorem \ref{BergmanWeakType} is new even in the unweighted setting ($\sigma=1$). In this case, Theorem \ref{BergmanWeakType} can be viewed as an extension of  McNeal's results of \cite{M1994} to domains with near minimal smoothness and also of the work of Lanzani and Stein in \cite{LS2012} to address the behavior at the $p=1$ endpoint. In fact, Theorem \ref{BergmanWeakType} and an interpolation argument imply the $L^p(D)$, $1<p<\infty$, boundedness result of \cite{LS2012} in the case of $D$ having $C^4$ boundary. With $B_1$ weights, Theorem \ref{BergmanWeakType} generalizes Bekoll\'e's endpoint weak-type result of \cite{B198182} to domains with near minimal smoothness and extends the work in \cite{WW2020} to address the $p=1$ endpoint.

The weak-type estimate of Theorem \ref{BergmanWeakType} implies some other useful endpoint bounds, generalizing results in \cite{DHZZ2001}. In particular, one has the following weighted Kolmogorov inequality:
\begin{cor}\label{KolmogorovBergman} 
If $D\subseteq \mathbb{C}^n$ is a strongly pseudoconvex bounded domain with $C^4$ boundary, $\sigma \in B_1$, and $0<p<1$, then the Bergman projection $\mathcal{B}$ extends boundedly from $L^1_\sigma(D)$ to $L^p_\sigma(D).$ That is, there exists $C>0$ such that
$$
\|\mathcal{B}f\|_{L_{\sigma}^p(D)}\leq C \|f\|_{L_{\sigma}^1(D)}
$$
for all $f \in L_{\sigma}^1(D)$.
\end{cor}
\noindent Additionally, one also gets the following Zygmund inequality as a corollary:
\begin{cor}\label{ZygmundBergman}
If $D\subseteq \mathbb{C}^n$ is a strongly pseudoconvex bounded domain with $C^4$ boundary and $\sigma \in B_1$, then the Bergman projection $\mathcal{B}$ extends boundedly from $(L \log^+ L)_{\sigma}(D)$ to $L_{\sigma}^1(D)$. That is, there exists $C>0$ such that
$$
    \|\mathcal{B}f\|_{L_{\sigma}^1(D)}\leq C\|f\|_{(L\log^+ L)_{\sigma}(D)}
$$
for all $f \in (L \log^+ L)_{\sigma}(D)$.
\end{cor}
\noindent Refer to Section \ref{BonusEstimates} for a precise definition of the Zygmund spaces $L \log^+ L$ and their norms. 



We also study the Cauchy-Szeg\H o projection on domains with near minimal smoothness. The Hardy space, $H^2(bD)$, is defined to be the following closure in $L^2(bD)$:
$$
    H^2(bD):= \overline{\{f \in L^2(bD): f=F\mid_{bD}, \,\, F\in \text{Hol}(D), \,\, \text{and} \,\, F \in C^0(\overline{D})\}}.
$$
Since $H^2(bD)$ is a closed subspace of $L^2(bD)$, there is an orthogonal projection from $L^2(bD)$ to $H^2(bD)$ which we call the Cauchy-Szeg\H o projection and denote by $\mathcal{S}$. We can view $\mathcal{S}$ as an integral operator via
$$
    \mathcal{S}f(z)=\int_{bD} K(z,w)f(w)\,dS(w),
$$
where $K$ is the reproducing kernel for $H^2(bD)$ and $S$ denotes the induced Lebesgue measure on $bD$.

Again, it is clear that $\mathcal{S}$ is a bounded operator on $L^2(bD)$. The bounds of $\mathcal{S}$ on $L^p(bD)$ for $1<p<\infty$ have a long history beginning in the case $D=\mathbb{D}$, where $\mathcal{S}$ is the Cauchy transform. In this case, the classical theorem of M. Riesz asserts that $\mathcal{S}$ acts as a bounded on $L^p(b\mathbb{D})$ for $1<p<\infty$. 
Recently, the $L^p(bD)$ bounds for $\mathcal{S}$ on domains with minimal smoothness were proved by Lanzani and Stein in \cite{LS2017}. In particular, they showed that if $D \subseteq \mathbb{C}^n$ is strongly pseudoconvex and bounded with $C^2$ boundary, then $\mathcal{S}$ extends as a bounded operator on $L^p(bD)$ for $1<p<\infty$. 

The characterization of weighted bounds for $\mathcal{S}$ in the case $D=\mathbb{D}$ is given by the $A_p$ condition of Hunt, Muckenhoupt, and Wheeden from \cite{HMW1973}. For $1<p<\infty$, a weight $\sigma$ satisfies the $A_p$ condition if 
$$
    [\sigma]_{A_p}:=\sup_{B}\left(\frac{1}{S(B)}\int_B \sigma\,dS\right)\left(\frac{1}{S(B)}\int_B \sigma^{-\frac{1}{p-1}}\,dS\right)^{p-1}<\infty,
$$
where supremum is taken over all quasi-balls (with respect to the quasi-metric defined in Section \ref{CauchySzegoSection}) $B\subseteq bD$. When $p=1$, we say $\sigma$ is an $A_1$ weight if
$$
    [\sigma]_{A_1}:=\sup_{B}\left(\frac{1}{S(B)}\int_B \sigma\,dS\right)\|\sigma^{-1}\|_{L^{\infty}(B)}<\infty.
$$
The bounds of $\mathcal{S}$ on $L_{\sigma}^p(bD)$ for $1<p<\infty$ and $\sigma \in A_p$ were recently established in the near minimal smoothness case where $D$ is a strongly pseudoconvex bounded domain with $C^3$ boundary by the second author and Wick in \cite{WW2020}.

The second main result of this paper is the weighted weak-type $(1,1)$ estimate for the Cauchy-Szeg\H o projection on domains with near minimal smoothness.

\begin{thm}\label{CauchySzegoWeakType}
If $D \subseteq \mathbb{C}^n$ is a strongly pseudoconvex bounded domain with $C^3$ boundary and $\sigma$ is an $A_1$ weight on $bD$, then the Cauchy-Szeg\H o projection $\mathcal{S}$ extends boundedly from $L_{\sigma}^1(bD)$ to $L_{\sigma}^{1,\infty}(bD)$. That is, there exists $C>0$ such that 
$$
    \|\mathcal{S}f\|_{L_{\sigma}^{1,\infty}(bD)}:=\sup_{\lambda>0}\lambda \sigma(\{z\in bD: |\mathcal{S}f(z)|>\lambda\})\leq C\|f\|_{L_{\sigma}^1(bD)}
$$
for all $f \in L_{\sigma}^1(bD)$.
\end{thm}

We remark that Theorem \ref{CauchySzegoWeakType} is new even in the unweighted setting. Theorem \ref{CauchySzegoWeakType} can be viewed as a weighted extension of the work of Lanzani and Stein in \cite{LS2012} and of the second author and Wick in \cite{WW2020} to address the behavior at the $p=1$ endpoint in the case of near minimal smoothness.

Similar to the case of the Bergman projection, we obtain a weighted Kolmogorov ineqaulity and a weighted Zygmund inequality for the Cauchy-Szeg\H{o} projection.
\begin{cor}\label{KolmogorovSzego} If $D\subseteq \mathbb{C}^n$ is a strongly pseudoconvex bounded domain with $C^3$ boundary, $\sigma \in A_1$, and $0<p<1$, then the Cauchy-Szeg\H{o} projection $\mathcal{S}$ extends boundedly from $L^1_\sigma(bD)$ to $L^p_\sigma(bD)$. That is, there exists $C>0$ such that
$$
    \|\mathcal{S}f\|_{L_{\sigma}^p(bD)} \leq C \|f\|_{L_{\sigma}^1(bD)}
$$
for all $f \in L_{\sigma}^1(bD)$.
\end{cor}
\begin{cor}\label{ZygmundSzego}
If $D\subseteq \mathbb{C}^n$ is a strongly pseudoconvex bounded domain with $C^3$ boundary and $\sigma \in A_1$, then the Cauchy-Szeg\H o projection $\mathcal{S}$ extends boundedly from $(L \log L)_{\sigma}(bD)$ to $L_{\sigma}^1(bD)$. That is, there exists $C>0$ such that
$$
    \|\mathcal{S}f\|_{L_{\sigma}^1(bD)}\leq C\|f\|_{(L\log L)_{\sigma}(bD)}
$$
for all $f \in (L \log L)_{\sigma}(bD)$.
\end{cor}
\noindent

Throughout this paper, we use the notation $A \lesssim B$ to mean $A\leq CB$ for some $C>0$ that could possibly depend on $n$, anything intrinsic to $D$, or $A_1$, $B_1$ weight characteristics. Although we will not keep track of constants depending on the weights, we will explicitly mention whenever their conditions are used. We say $A \approx B$ if both $A\lesssim B$ and $B \lesssim A$. We use the notation $\langle f\rangle_{E, \mu}$ to denote the average $\frac{1}{\mu(E)}\int_E f\,d\mu$. We just write $\langle f\rangle_E$ to represent this average when $\mu$ is Lebesgue measure in Section \ref{BergmanSection} and induced Lebesgue measure on $bD$ in Section \ref{CauchySzegoSection}. For a weight $\sigma$, we write $\sigma(E)$ to represent $\int_E \sigma\,dV$ in Section \ref{BergmanSection} and to represent $\int_E \sigma\,dS$ in Section \ref{CauchySzegoSection}.

The paper is organized as follows. In Section \ref{BergmanSection}, we prove the weighted weak-type $(1,1)$ inequality for the Bergman projection, Theorem \ref{BergmanWeakType}. In Section \ref{CauchySzegoSection}, we prove the weighted weak-type $(1,1)$ estimate for the Cauchy-Szeg\H o projection, Theorem \ref{CauchySzegoWeakType}. Finally in Section \ref{BonusEstimates}, we obtain Corollaries \ref{KolmogorovBergman}, \ref{ZygmundBergman}, \ref{KolmogorovSzego}, and \ref{ZygmundSzego} via general principles.

The authors would like to thank Brett Wick for inspiring discussions and feedback.

\section{The Bergman Projection}\label{BergmanSection}
\subsection{Setup}
Let $D \subseteq \mathbb{C}^n$ be a strongly pseudoconvex bounded domain with $C^4$ boundary. This means that there exists a strictly plurisubharmonic, $C^4$ defining function $\rho$ with $D=\{z\in \mathbb{C}^n:\rho(z)<0\}$ and $\nabla{\rho}(z) \neq 0$ for $z \in bD$.

Our general approach is to construct an auxiliary operator $\mathcal{T}$ that produces and reproduces holomorphic functions. We follow the same construction as in \cites{LS2012,R1986}, first constructing an operator $\mathcal{T}_1$ that reproduces (but does not produce) holomorphic functions, and then introducing an operator $\mathcal{T}_2$ to correct it. The operator $\mathcal T$ is taken to be $\mathcal T_1 +\mathcal T_2$.

To construct $\mathcal{T}_1$, we use the holomorphic integral representations known as Cauchy-Fantappi\'{e} integrals. For $w \in D$, we define the Levi polynomial at $w$ as follows:
$$
    P_w(z):= \sum_{j=1}^{n}\frac{\partial \rho}{\partial w_j}(w)(z_j-w_j)+\frac{1}{2}\sum_{j,k=1}^n\frac{\partial^2 \rho}{\partial w_j \partial w_k}(w)(z_j-w_j)(z_k-w_k).
$$ 
We have to modify the Levi polynomial slightly to make it usable for our purposes. In particular, using the strict pseudoconvexity of $D$, it is possible to choose a smooth cutoff function $\chi=\chi(z,w)$ with $\chi \equiv 1$ when $|z-w|<\delta/2$ and $\chi \equiv 0$ when $|z-w| >\delta$ for a small constant $\delta>0$ such that the function
$$
    g(z,w):= -\chi P_w(z)+(1-\chi)|z-w|^2$$ satisfies 
$$
    \text{Re} \,g(z,w) \gtrsim -\rho(w)-\rho(z)+|z-w|^2.
$$

Define the $(1,0)$ form in $w$, $G(z,w)$, as follows:
$$
    G(z,w):= \chi\left(\sum_{j=1}^n \frac{\partial \rho}{\partial w_j}(w)\,dw_j+\frac{1}{2}\sum_{j,k=1}^n\frac{\partial^2 \rho}{\partial w_j \partial w_k}(w)(z_k-w_k)\,dw_j\right) + (1-\chi)\sum_{j=1}^n(\overline{w}_j - \overline{z}_j)\,dw_j.
$$
Notice that 
$$
    \langle G(z,w), w-z\rangle = g(z,w)+\rho(w),
$$
where $\langle \cdot , \cdot \rangle$ denotes the action of a $(1,0)$ form on a vector in $\mathbb{C}^n$.
Now we take 
$$
    \eta(z,w):= \frac{G(z,w)}{g(z,w)}
$$
and define
$$
    \mathcal{T}_1f(z):=\frac{1}{(2 \pi \mathrm{i})^n}\int_D (\overline{\partial}_w\eta)^nf(w),
$$
where the exponent $n$ denotes the wedge product taken $n$ times.

With this definition, one can show that $\mathcal{T}_1$ is majorized by a positive operator $\Gamma$ which can be interpreted as a Calder\'{o}n-Zygmund operator. This approach is taken in the proof of Proposition \ref{BergmanAuxiliaryWeakType} below. Notice also that the kernel of $\mathcal T_1$ is continuous on $\overline{D} \times \overline{D}$ away from the boundary diagonal $\{(z,z):z \in bD\}$. Moreover, the following is proven in \cite{LS2012}:

\begin{lemma}\label{reproduce holomorphic}
If $f \in L^1(D)$ is holomorphic on $D$, then for all $z \in D$,
$$\mathcal T_1f(z)=f(z).$$
\end{lemma}

Lemma \ref{reproduce holomorphic} says that $\mathcal T_1$ reproduces holomorphic functions. However, in general $\mathcal T_1$ does not produce holomorphic functions from $L^1$ data since its kernel is not holomorphic in the variable $z$. 
We introduce a correction operator, $\mathcal T_2$, to overcome this difficulty. The operator $\mathcal T_2$ is constructed by solving a $\bar \partial$ problem on a smoothly bounded, strongly pseudoconvex domain $\Omega$ that contains $D$. These details are unimportant for our purposes; we only need the following from \cite{LS2012}:
\begin{lemma}\label{correction operator}
There exists an integral operator 
$$\mathcal T_2 f(z):= \int_{D} K_2(z,w) f(w)\, dV(w)$$
with continuous kernel $K_2(z,w)$ on $\overline{D} \times \overline{D}$ so that the operator $\mathcal T:= \mathcal T_1+ \mathcal T_2$ satisfies the following properties:
\begin{enumerate}
    \item If $f \in L^1(D)$, then $\mathcal Tf$ is holomorphic on $D$.
    \item If $f \in L^1(D)$ and $f$ is holomorphic on $D$, then $\mathcal Tf(z)=f(z)$ for $z \in D.$
\end{enumerate}
\end{lemma}
\noindent It is important to note that Lemma \ref{correction operator} together with the definition of $\mathcal{T}_1$ implies that the kernel of $\mathcal{T}$ is continuous on $\overline{D} \times \overline{D}$ away from the boundary diagonal.

We have constructed an auxiliary operator $\mathcal{T}$ that produces and reproduces holomorphic functions. Since $\mathcal{B}$ also produces and reproduces holomorphic functions, we arrive at the following operator equations that hold on $L^2(D):$
$$
    \mathcal T \mathcal B=\mathcal B \quad\text{and}\quad \mathcal B \mathcal T=\mathcal T.
$$ 
Taking adjoints in the first identity, subtracting from the second, and some rearrangement yields the familiar Kerzman-Stein equation:
\begin{align}\label{KerzmanStein}
    \mathcal B (I-(\mathcal T^*-\mathcal T))=\mathcal T.
\end{align}

The proof of Theorem \ref{BergmanWeakType} follows easily from the following two facts.
\begin{prop}\label{BergmanInversion}
If $\sigma$ is a $B_1$ weight, then the operator $I-(\mathcal T^*-\mathcal T)$ is invertible on $L_{\sigma}^1(D)$.
\end{prop}
\begin{prop}\label{BergmanAuxiliaryWeakType}
If $\sigma$ is a $B_1$ weight, then $\mathcal{T}$ maps $L_{\sigma}^1(D)$ to $L_{\sigma}^{1,\infty}(D)$ boundedly.
\end{prop}

\begin{proof}[Proof of Theorem \ref{BergmanWeakType}]
Using Proposition \ref{BergmanInversion}, we may rewrite \eqref{KerzmanStein} as
$$
    \mathcal B= \mathcal{T} (I-(\mathcal T^*-\mathcal T))^{-1}.
$$
The bound of $\mathcal{B}$ from $L_{\sigma}^1(D)$ to $L_{\sigma}^{1,\infty}(D)$ follows from Proposition \ref{BergmanInversion} and Proposition \ref{BergmanAuxiliaryWeakType}.
\end{proof}

The remainder of this section is devoted to proving Proposition \ref{BergmanInversion} and Proposition \ref{BergmanAuxiliaryWeakType}. Proposition \ref{BergmanInversion} will follow from the spectral theorem for compact operators on a Banach space. In particular, we will show that $\mathcal T^*-\mathcal T$ is compact on $L_{\sigma}^1(D)$ and also that $1$ is not an eigenvalue of $\mathcal T^*-\mathcal T$ on $L_{\sigma}^1(D)$. Proposition \ref{BergmanAuxiliaryWeakType} relies on methods from Calder\'on-Zygmund theory reminiscent of the ideas in \cite{B198182}.



The arguments in \cites{HWW2020,M2003,M1994,WW2020} make use of an appropriately constructed quasi-metric $d$ that reflects the geometry of the boundary. Technically, the quasi-metric $D$ is only defined for points $z,w$ sufficiently close to the boundary, but we will abuse notation and define objects  as if $d$ were defined globally. This reduction is possible because the kernels of all the relevant operators are uniformly continuous on compact subsets of $\overline{D} \times \overline{D}$ off the boundary diagonal and all the necessary properties will hold for trivial reasons. 

The quasi-metric $d$ locally satisfies:
$$
    d(z,w) \approx |z_1-w_1|+\sum_{j=2}^{n}|z_j-w_j|^2,
$$ 
where the coordinates $z_j$ and $w_j$ are taken in a special holomorphic coordinate system centered at $w$. The coordinate function $z_1$ corresponds to the radial direction, while $z_2,\dots,z_n$ describe the complex tangential directions. In \cite{M2003}, these coordinates were used to obtain favorable estimates on the Bergman kernel for smoothly bounded domains $D$, and in \cite{WW2020} they were used in the case when $D$ has less boundary regularity. 

We use the constant $c>0$ to denote the implicit constant in the triangle inequality for $d$: 
$$
    d(z,w)\leq cd(z,\zeta)+cd(\zeta,w).
$$
We denote balls with respect to this quasi-metric by $B(z,r):=\{w\in D: d(z,w)<r\}$. If $B$ is a quasi-ball, then its center and radius are represented by $c(B)$ and $r(B)$ respectively, meaning $B=\{w \in D: d(c(B),w)<r(B)\}$. We also write $kB$ to denote the $k$-fold dilate of $B$, that is $kB:=\{w\in D: d(c(B),w)<kr(B)\}$. 

Importantly, the triple $(D,d,V)$ forms a space of homogeneous type in the sense of Coifman and Weiss introduced in \cite{CW1971}. In particular, $V$ satisfies the following growth condition with respect to quasi-balls induced by $d$: 
\begin{align}\label{Vgrowth}
    V(B(z,r)) \approx r^{n+1}
\end{align}
for all $z \in D$ and $r>0$. Moreover, the distance function $d$ can be extended to $\overline{D} \times \overline{D}$ and we may also define $d(z,bD):= \inf_{w \in bD}d(z,w)$, see \cite{HWW2020}. Notice that for a $B_1$ weight $\sigma$, $\sigma\,dV$ also satisfies a particular doubling property for quasi-balls close to the boundary: 
$$
    \sigma(B(z,2r))\lesssim \left(\inf_{w\in B(z,2r)}\sigma(w)\right)V(B(z,2r)) \lesssim  \sigma(B(z,r))
$$
for any $z \in D$ and $r>0$ such that $r>kd(z,bD)$ for some absolute $k>0$ (the first inequality above depends on $[\sigma]_{B_1}$). For sets $E,F\subseteq \overline{D}$, we write $d(E,F):=\inf_{\substack{z \in E \\ w \in F}} d(z,w)$.

We work with a maximal operator $\mathcal{M}$ adapted to our setting. For locally integrable $f$, define
$$
    \mathcal{M}f(z):=\sup_{\substack{B(w,r)\ni z \\ r>d(w,bD)}} \langle |f|\rangle_{B(w,r)}.
$$
Note that a weight $\sigma$ is in $B_1$ if and only if $\mathcal{M}\sigma(z)\lesssim \sigma(z)$ for almost every $z\in D$. 

\subsection{Inversion of the ``mild" operator}


To deduce the compactness of $\mathcal{T}^* -\mathcal{T}$, we use a more general result which follows from \cite[Corollary 4.1]{E1995}. In the following lemma, $\mathcal{K}$ is an integral operator given by
$$
    \mathcal{K}f(x)= \int_{X} k(x,y) f(y)\mathop{d \mu(y)}
$$ 
and $k_y(x)=k(x,y).$ 
\begin{lemma} \label{compact operator}
Let $(X,\mu)$ be a positive measure space. Suppose that $k:X\times X \rightarrow \mathbb{R}$ is a measurable function such that 
$\|\int_X k(x,\cdot)\,d\mu(x)\|_{L^{\infty}(X,\mu)}<\infty$. If the set $\{k_y\}_{y\in X}$ is relatively compact in $L^1(X,\mu)$, then $\mathcal K$ and $\mathcal K^*$ are compact operators on $L^1(X,\mu)$ and $L^\infty(X,\mu)$ respectively. 
\end{lemma}

To justify the relative compactness of $\{k_y\}$ in our application of Lemma \ref{compact operator}, we use the following characterization for relatively compact sets, which can be viewed as a generalization of the classical Riesz-Kolmogorov theorem.
\begin{lemma}\label{relatively compact}
Let $\mu$ be a finite Borel measure on $X$ such that $\inf_{x\in X} \mu(B(x,r)) >0$ for any $r>0$ and let $1\leq p<\infty$. If $K \subseteq L^p(X,\mu)$ is a bounded set satisfying
$$
    \lim_{r\rightarrow 0^+} \sup_{f \in K}\int_X |f(x)-\langle f\rangle_{B(x,r),\mu}|^p\,d\mu(x)=0,
$$
then $K$ is relatively compact in $L^p(X,\mu)$. 
\end{lemma}
\noindent Lemma \ref{relatively compact} was originally stated for the case of metric spaces in \cite{K1999}, but we will need a version from \cite[Lemma 1]{IK2009} in the case where we only have a quasi-metric.

We next apply Lemma \ref{compact operator} and Lemma \ref{relatively compact} to prove the following result.
\begin{lemma}\label{BergmanLCompactness} 
If $\sigma$ is a $B_1$ weight, then the  operator $\mathcal T^*-\mathcal T$ is compact on $L^1_\sigma(D)$.
\end{lemma}
\begin{proof}
First, we note that $\sigma\,dV$ is a finite Borel measure on $D$. Using the $B_1$ condition and the fact that $B(z,R)=D$ for $z \in D$ and sufficiently large $R$, one has 
$$
\sigma(D)\lesssim \left(\inf_{w\in D}\sigma(w)\right)V(D).
$$
The infimum condition on the measure $\sigma dV$ can be verified using a compactness argument and the fact that $B(z,r)$ contains a Euclidean ball with radius comparable to $r^{1/2}$, which was proved in \cite[Proposition 3.5]{WW2020}. Let $k(z,w)$ denote the kernel of $\mathcal{T}^*-\mathcal{T}$ with respect to Lebesgue measure. The following key properties of $k(z,w)$ are proved in \cite[Lemma 3.14]{WW2020}:
$$
    |k(z,w)| \lesssim |g(z,w)|^{-\left(n+\frac{1}{2}\right)} \lesssim d(z,w)^{-\left(n+\frac{1}{2}\right)} 
$$
as well as
$$
    |k(z,w)| \lesssim \min\left\{ d(z,bD)^{-\left(n+\frac{1}{2}\right)}, d(w,bD)^{-\left(n+\frac{1}{2}\right)}\right\}.
$$
Here, the assumption that the boundary of $D$ is of class $C^4$ is in fact crucial. Let $\tilde{k}(z,w)$ denote the kernel of $\mathcal{T}^*-\mathcal{T}$ with respect to the weighted measure $\sigma\,dV$ and notice $\tilde{k}(z,w)=k(z,w) \sigma(w)^{-1}.$ 

We claim that there exists $M>0$ such that $\sup_{w \in D}\int_{D}|\tilde{k}(z,w)|\sigma(z)\,dV(z)<M.$ To see this, fix $w \in D$ and integrate over dyadic annuli, choosing $R$ so that $B(w,R)=D$ and letting $N$ be the largest positive integer such that $B(w,2^{-N}R)$ meets the boundary of $D$. We use the above control of $|k(z,w)|$ and \eqref{Vgrowth} to obtain
\begin{align*}
\int_{D}|\tilde{k}(z,w)|\sigma(z)\,dV(z) & =  \sigma(w)^{-1} \int_{D} |k(z,w)| \sigma(z)\,dV(z)\\
& =  \sigma(w)^{-1} \sum_{j=0}^{N} \int_{B(w,2^{-j}R) \setminus B(w,2^{-(j+1)}R)} d(z,w)^{-\left(n+\frac{1}{2}\right)}\sigma(z) \mathop{dV(z)}\\
& +  \sigma(w)^{-1}\int_{B(w,2^{-(N+1)}R)} d(w,bD)^{-(n+\frac{1}{2})} \sigma(z) \mathop{dV(z)}\\
 & \lesssim  \sigma(w)^{-1} \sum_{j=0}^{N} \frac{2^{-j/2}R^{1/2}}{V(B(w,2^{-j}R))}\int_{B(w,2^{-j}R)} \sigma(z) \mathop{dV(z)}\\
 & +     \sigma(w)^{-1}\frac{d(w,bD)^{1/2}}{V(B(w,d(w,bD)))}\int_{B(w,d(w,bD))} \sigma(z)\,dV(z)\\
& \leq  \sigma(w)^{-1}\sum_{j=0}^{N} 2^{-j/2} R^{1/2} \mathcal{M}\sigma(w)+ \sigma(w)^{-1} d(w,bD)^{1/2} \mathcal{M}\sigma(w)\\
& \lesssim  \sigma(w)^{-1} (R^{1/2}+d(w,bD)^{1/2}) \mathcal{M}\sigma(w)\\
& \lesssim  R^{1/2}.
\end{align*}
Note that we used the $B_1$ condition in the last line above. All the implicit constants are independent of $w$, and $R$ is also independent of $w$ since we can just take $R$ to be the diameter of $D$ in the quasi-metric. This establishes the claim. Notice that this argument also shows that replacing the region of integration by a quasi-ball $B(w,\delta)$ yields
\begin{align}
\int_{B(w,\delta)}|\tilde{k}(z,w)|\sigma(z)\,dV(z)\lesssim \delta^{1/2} +d(w,bD)^{1/2}, \label{1}
\end{align}
where the implicit constant is independent of $w$.

Now we must show the crucial condition
$$
    \lim_{r\rightarrow 0^+} \sup_{w \in D} \sigma(w)^{-1} \int_D |k_w(z)-\langle k_w\rangle_{B(z,r),\sigma dV}|\sigma(z)\,dV(z)=0,
$$ where $k_w(z)=k(z,w).$ 
Fix $\varepsilon>0$, $w \in D$, and let $\delta>0$ and $0< r<\delta$ be constants to be fixed later. We emphasize all constants obtained will ultimately be independent of $w$. 

Let $G:=\{z \in D: d(z,w)\geq \delta \text{ or } d(z,bD) \geq \delta \}$. We will first estimate
$$
    \sigma^{-1}(w)\int_{G}|k_w(z)-\langle k_w \rangle_{B(z,r),\sigma dV}|\sigma(z)\,dV(z).
$$ 
Recall that the kernel function $k(z,w)$ is uniformly continuous on compact subsets off the boundary diagonal, so in particular the function $k_w(z)$ is uniformly continuous on $G$ with a modulus of continuity independent of $w$. We can choose $r$ sufficiently small relative to $\delta$ and independent of $w$ so that we have $|k_w(z)-\langle k_w \rangle_{B(z,r),\sigma dV}|< \varepsilon$ for $z \in G$ and hence,

\begin{align*}
\sigma(w)^{-1}\int_{G}|k_w(z)-\langle k_w \rangle_{B(z,r),\sigma dV}|\sigma(z)\,dV(z) & \leq  \varepsilon\sigma(w)^{-1}\int_{D} \sigma(z)\,dV(z)\\
& \lesssim  \varepsilon \sigma(w)^{-1} \mathcal{M} \sigma(w)\\
& \lesssim  \varepsilon
\end{align*}
as required. We used the $B_1$ condition of $\sigma$ in the last inequality above.

Now we need to estimate the integral on $D \setminus G$. Note $D \setminus G =B(w, \delta) \cap A$, where $A:=\{z:d(z,bD)< r \}$.
We have
\begin{align*}
    \sigma(w)^{-1}\int_{D \setminus G}|k_w(z)-\langle k_w \rangle_{B(z,r),\sigma dV}|\sigma(z)\,dV(z)\leq &\,\, \sigma(w)^{-1}\left( \int_{D \setminus G}|k_w(z)|\sigma(z)\,dV(z)\right.\\
    &+ \left. \int_{D \setminus G} |\langle k_w \rangle_{B(z,r),\sigma dV}|\sigma(z)\,dV(z)\right).
\end{align*}
By \eqref{1}, it is easy to deduce 
$$
    \sigma(w)^{-1} \int_{D \setminus G}|k_w(z)|\sigma(z)\,dV(z) \lesssim \delta^{1/2}.
$$ 

We will also show 
$$
    \sigma(w)^{-1} \int_{D \setminus G} |\langle k_w \rangle_{B(z,r),\sigma dV}|\sigma(z)\,dV(z) \lesssim \delta^{1/2}
$$ 
using similar methods. We consider two separate regions of integration based on the relative positions of $z$ and $w$. First, suppose that $cr<\frac{1}{2} d(z,w)$. One can show that if $\zeta \in B(z,r)$, then $d(z,w) \lesssim d(\zeta,w)$ with an implicit constant independent of $z$ and $w$. We then estimate
\begin{align*}
 & \sigma(w)^{-1} \int_{(B(w,\delta) \setminus B(w,2cr)) \cap A} |\langle k_w \rangle_{B(z,r),\sigma dV}|\sigma(z)\,dV(z)\\
 & \leq \sigma(w)^{-1} \int_{(B(w,\delta) \setminus B(w,2cr)) \cap A}\frac{1}{\sigma(B(z,r))} \int_{B(z,r)} d(\zeta,w)^{-\left(n+\frac{1}{2}\right)} \sigma(\zeta) \mathop{dV(\zeta)} \sigma(z) \mathop{d V(z)}\\
 & \lesssim \sigma(w)^{-1} \int_{B(w,\delta) \cap A}  d(z,w)^{-\left(n+\frac{1}{2}\right)}  \sigma(z) \mathop{d V(z)} \\
 & \lesssim \delta^{1/2}
\end{align*}
as before. We have used the $B_1$ condition of $\sigma$ in the third inequality above. 

On the other hand, if $d(z,w) \leq 2cr$, then $B(z,r) \subseteq B(w, Cr)$ and $B(w,r) \subseteq B(z,Cr)$, where $C=2c^2+c$. We first consider a further subcase where $d(w,bD) <r$. In this case, note $d(z,bD) \lesssim r$ on this set as well by the quasi-triangle inequality. Thus, we calculate:

\begin{align*}
 & \sigma(w)^{-1} \int_{B(w,2cr) \cap A}\frac{1}{\sigma(B(z,r))} \int_{B(z,r)} d(\zeta,w)^{-\left(n+\frac{1}{2}\right)} \sigma(\zeta) \mathop{dV(\zeta)} \sigma(z) \mathop{d V(z)}\\
 & \leq \sigma(w)^{-1} \frac{1}{\sigma(B(w,r))}\int_{B(w,2cr)\cap A} \frac{\sigma(B(z,Cr))}{\sigma(B(z,r))} \int_{B(w,Cr)} d(\zeta,w)^{-\left(n+\frac{1}{2}\right)} \sigma(\zeta) \mathop{dV(\zeta)} \sigma(z) \mathop{d V(z)}\\
 & \lesssim \delta^{1/2} \frac{1}{\sigma(B(w,r))}\int_{B(w,2cr) \cap A} \frac{\sigma(B(z,Cr))}{\sigma(B(z,r))}  \sigma(z) \mathop{d V(z)}\\
 & \lesssim \delta^{1/2}
 \end{align*}
using the $B_1$ condition in the second inequality and the doubling property of $\sigma$ in the third inequality. For the second subcase, suppose $d(w,bD) \geq r$ and note that we still assume $d(z,w) \leq 2 cr$, so we in fact have $d(w,bD)^{-(n+1/2)} \lesssim d(z,w)^{-(n+1/2)}$. We estimate
\begin{align*}
 & \sigma(w)^{-1} \int_{ B(w,2cr) \cap A} |\langle k_w \rangle_{B(z,r),\sigma dV}|\mathop{d\sigma(z)}\\
 & \leq \sigma(w)^{-1} \int_{ B(w,2cr) \cap A}\frac{1}{\sigma(B(z,r))} \int_{B(z,r)} d(w,bD)^{-\left(n+\frac{1}{2}\right)} \sigma(\zeta) \mathop{dV(\zeta)} \sigma(z) \mathop{d V(z)}\\
 & \lesssim \sigma(w)^{-1} \int_{B(w,\delta) \cap A}  d(z,w)^{-\left(n+\frac{1}{2}\right)}  \sigma(z) \mathop{d V(z)} \\
 & \lesssim \delta^{1/2},
\end{align*}
where we have used the $B_1$ condition in the third inequality.

Thus, we obtain
$$
    \sigma(w)^{-1}\int_{D \setminus G}|k_w(z)-\langle k_w \rangle_{B(z,r),\sigma dV}|\sigma(z)\,dV(z)\lesssim \delta^{1/2}
$$
with an independent implicit constant. This can be made less than $\varepsilon$ by making an appropriately small choice of $\delta$, completing the proof.
\end{proof}

We need the following lemma to conclude that $(I-(\mathcal{T}^*-\mathcal{T}))$ is invertible on $L^1_{\sigma}(D)$. 
\begin{lemma}\label{Bergmanspectrum}
If $\sigma \in B_1$, the number $1$ is not an eigenvalue of $\mathcal{T}^*-\mathcal{T}$ considered as an operator on $L^1_{\sigma}(D)$.
\end{lemma}
\begin{proof}
The proof proceeds in the same way as \cite[Corollary 3.17]{WW2020}. In particular, it was proved in \cite[Proposition 3.16]{WW2020} that there exists $\varepsilon>0$ so that $\mathcal{T}^*-\mathcal{T}$ maps $L^p(D)$ to $L^{p+\varepsilon}(D)$ boundedly for $p \geq 1.$ Thus, if $1$ were an eigenvalue for $\mathcal{T}^*-\mathcal{T}$ with eigenvector $f \in L^1_{\sigma} (D)$, then we would have 
$$\|f\|_{L^{1+\varepsilon}(D)}=\|(\mathcal T^*-\mathcal T)f\|_{L^{1+\varepsilon}(D)} \lesssim \|f\|_{L^1(D)} \lesssim \|f\|_{L^1_{\sigma}(D)} ,$$
noting that a weight in $B_1$ is bounded below. If we repeat this argument a second time, we get $f \in L^{1+2\varepsilon}(D).$ In fact, we can iterate arbitrarily many times to obtain $f\in L^p(D)$ for all $p \geq 1.$ In particular, $f \in L^2(D).$ This is a contradiction because $1$ is not an eigenvalue of $\mathcal{T}^*-\mathcal{T}$ on $L^2(D)$, since all of these eigenvalues are purely imaginary.
\end{proof}

\begin{proof}[Proof of Proposition \ref{BergmanInversion}]
This follows immediately from Lemma \ref{BergmanLCompactness} and Lemma \ref{Bergmanspectrum} using the spectral theorem for compact operators.
\end{proof}

\subsection{Weak-type estimate for the auxiliary operator}

To show the weighted weak-type $(1,1)$ property for $\mathcal{T}$, we first prove the analogous bound for our maximal operator $\mathcal{M}$.

\begin{lemma}\label{MaximalWeakType}
If $\sigma$ is a $B_1$ weight, then $\mathcal{M}$ maps $L_{\sigma}^1(D)$ into $L_{\sigma}^{1,\infty}(D)$ boundedly. 
\end{lemma}

\begin{proof}
It suffices to prove the estimate for the centered version of $\mathcal{M}$, 
$$
    \widetilde{\mathcal{M}}f(z):=\sup_{r>d(z,bD)}\langle |f| \rangle_{B(z,r)},
$$
since we have the pointwise equivalence $\widetilde{\mathcal{M}}f\leq \mathcal{M}f\lesssim \widetilde{\mathcal{M}}f$. Indeed, the first inequality is clear, and the second is justified by the fact that $\langle |f|\rangle_{B} \lesssim \langle |f|\rangle_{B(z,2cr(B))}$
for any $z\in D$ and quasi-ball $B$ containing $z$. 

Let $f \in L_{\sigma}^1(D)$, $\lambda>0$, and $E_{\lambda}:=\{\widetilde{\mathcal{M}}f>\lambda\}$. We show that 
$$
    \sigma\left(E_{\lambda}\right)\lesssim \frac{1}{\lambda}\|f\|_{L_{\sigma}^1(D)}.
$$
For each $z \in E_{\lambda}$, let $B_z$ be a quasi-ball centered at $z$ such that $r(B_z)>d(z,bD)$ and $\langle |f|\rangle_{B_z} > \lambda$. Apply a Vitali-type lemma to obtain a subcollection $\{B_{j}\}_{j=1}^{\infty}$ of $\{B_z\}_{z\in E_{\lambda}}$ consisting of pairwise disjoint quasi-balls such that there exists $R\ge 1$ with $E_{\lambda} \subseteq \bigcup_{j=1}^{\infty}RB_j$. Use the doubling property of $\sigma$, the $B_1$ property of $\sigma$, and the selection property of the $B_j$ to conclude
$$
    \sigma\left(E_{\lambda}\right)\leq \sum_{j=1}^{\infty}\sigma(RB_j) \lesssim \sum_{j=1}^{\infty}\sigma(B_j)
$$
$$
    \lesssim\sum_{j=1}^{\infty} \left(\frac{1}{\|\sigma^{-1}\|_{L^{\infty}(B_j)}}\right)V(B_j)< \sum_{j=1}^{\infty}\left(\inf_{w \in B_j} \sigma(w)\right)\frac{1}{\lambda}\int_{B_j}|f|\,dV\leq \frac{1}{\lambda}\|f\|_{L_{\sigma}^1(D)}.
$$
\end{proof}

For $k \in (0,1)$, define $B_k(z):=B(z,kd(z,bD))$ and
$$
    R_kf(z):=\langle f\rangle_{B_k(z)}.
$$
The following was proved in \cite[Lemma 3.4]{HWW2020}. Note that in the setting of \cite{HWW2020}, $D$ had smooth boundary, but that assumption was not needed to establish the following lemma. 
\begin{lemma}\label{RkRk'}
For all $k \in \left(0,\frac{1}{2c}\right)$ and all nonnegative $f,g \in L^1(D)$, we have
$$
    \int_Df(R_kg)\,dV\lesssim \int_D (R_{k'}f)g\,dV,
$$
where $k':=\frac{ck}{1-ck}$.
\end{lemma}

\begin{lemma}\label{gisconstant}
There exists $k>0$ such that $|g(z,w)|\approx|g(z,w')|$ for all $z,w,w'\in D$ satisfying $d(w,w')\leq kd(w,bD)$.
\end{lemma}
\begin{proof}
From the proof of \cite[Lemma 3.12]{WW2020}, we know $|g(z,w)|\approx |g(z,w')|$ whenever $d(w,w')\leq Cd(z,w)$, where $C>0$ is an absolute constant. If $d(w,bD) < \frac{C}{k}d(z,w)$, then $d(w,w') \leq kd(w,bD)<Cd(z,w)$, and hence $|g(z,w)|\approx|g(z,w')|$. 

We may now assume that $d(z,w)\leq \frac{k}{C}d(w,bD)$. In this case, we use the triangle inequality, the fact that $|g(z,w)-g(z,w')|\lesssim d(w,w')^{\frac{1}{2}}d(z,w)^{\frac{1}{2}}+d(w,w')$ (which was proven in \cite[Lemma 3.12]{WW2020}), and the assumptions to get
\begin{align*}
    |g(z,w)|&\leq |g(z,w)-g(z,w')|+|g(z,w')|\\
    &\lesssim d(w,w')^{\frac{1}{2}}d(z,w)^{\frac{1}{2}}+d(w,w')+|g(z,w')|\\
    &\leq\left(\frac{k}{C^{1/2}}+k\right)d(w,bD)+|g(z,w')|.
\end{align*}
Now using the triangle inequality and the hypothesis, we have
$$
    d(w,bD)\leq cd(w,w')+cd(w',bD)\leq ckd(w,bD)+cd(w',bD).
$$
Choosing $k$ sufficiently small, the above line implies $d(w,bD) \lesssim d(w',bD)$, and so
$$
    |g(z,w)|\lesssim d(w',bD)+|g(z,w')|.
$$
Again referring to \cite[Lemma 3.12]{WW2020}, we have $d(w',bD)\lesssim |g(z,w')|$, and we conclude 
$$
    |g(z,w)|\lesssim |g(z,w')|.
$$
A symmetric argument proves the reverse inequality, establishing the lemma.
\end{proof}

The following lemma is a modified version of the Calder\'on-Zygmund decomposition.

\begin{lemma}\label{CZD}
For any $\lambda >0$, $k \in (0,1)$, and nonnegative $f \in L^1(D)$, we can write $f\approx f_1+f_2$, where
\begin{enumerate}
    \item $R_kf_1 \lesssim \lambda$,
    \item there exists a countable collection of almost disjoint quasi-balls $\mathcal{F}$ such that $r(B) \ge \frac{1}{2}d(B,bD)$ for each $B\in\mathcal{F}$ and $f_2 \approx \sum_{B\in\mathcal{F}}f_{2,B}$ where the $f_{2,B}$ are supported in $B$ with $\langle |f_{2,B}|\rangle_B\leq \lambda$, and
    \item $\sum_{B\in\mathcal{F}} \sigma(B)\lesssim \frac{1}{\lambda}\|f\|_{L_{\sigma}^{1}(D)}$.
\end{enumerate}
\end{lemma}

\begin{proof}
Apply a Whitney decomposition to write $$
    \{\mathcal{M}f>\lambda\}=\bigcup_{B\in\mathcal{F}'}B,
$$ 
where $\mathcal{F}'$ is a countable collection of almost disjoint quasi-balls for which there exists $K>1$ such that $KB\cap \{\mathcal{M}f\leq\lambda\}\neq \emptyset$ for all $B \in \mathcal{F}'$. We take
$$
    \mathcal{F}:=\left\{B \in \mathcal{F}': r(B) \ge \frac{1}{2}d(B,bD)\right\}.
$$ 
Put
$$
    f_1:= f\chi_{\{\mathcal{M}f\leq \lambda\}\,\cup \,\bigcup_{B \in \mathcal{F}'\setminus\mathcal{F}}B}
\quad\text{and}\quad
    f_2:= f\chi_{\bigcup_{B \in \mathcal{F}}B}.
$$
Clearly, $f\approx f_1+f_2$.

To show (1), we first claim that $R_kf_1(z) \lesssim \mathcal{M}f_1(z)$ for any $z \in D$. Indeed, since the radius of $\frac{k+1}{k}B_k(z)$ is greater than $d(z,bD)$ and using \eqref{Vgrowth}, we have
$$
    R_kf_1(z)=\langle f_1 \rangle_{B_k(z)}\lesssim \langle |f_1|\rangle_{\frac{k+1}{k}B_k(z)}\leq \mathcal{M}f_1(z).
$$
Therefore it is enough to prove $\mathcal{M}f_1\lesssim \lambda$. To this end, fix $z \in D$ and let $B_0$ be a quasi-ball containing $z$ that intersects $bD$. If either $B_0 \cap \{\mathcal{M}f\leq \lambda\}\neq \emptyset$ or if $f\equiv 0$ on $B_0$, then $\langle |f_1|\rangle_{B_0} \leq \lambda$. Otherwise, $B_0 \cap B \neq \emptyset$ for some $B \in \mathcal{F}' \setminus \mathcal{F}$. Notice that $CB_0 \supseteq KB$ with $C=c^3(K+1)+c$, since $d(c(B_0),bD)< r(B_0)$ and $r(B)< \frac{1}{2}d(B,bD)$. Since $KB \cap \{\mathcal{M}f\leq \lambda\}\neq \emptyset$ and using \eqref{Vgrowth}, we have
$$
\langle |f_1|\rangle_{B_0}\lesssim \langle |f_1|\rangle_{CB_0}\leq \lambda.
$$
Therefore (1) holds. 

For (2), note that the properties of the collection $\mathcal{F}$ are satisfied by construction. Take $f_{2,B}:=f\chi_{B}$ for $B \in \mathcal{F}$. Since $KB \cap \{\mathcal{M}f\leq\lambda\}\neq \emptyset$, we have $\langle |f_{2,B}|\rangle_B\lesssim \lambda$. 

Finally, (3) follows from the almost disjointness of the quasi-balls in $\mathcal{F}$ and Lemma \ref{MaximalWeakType}
$$
    \sum_{B\in\mathcal{F}} \sigma(B)\lesssim \sigma\left(\bigcup_{B\in\mathcal{F}}B\right)\leq \sigma(\{\mathcal{M}f>\lambda\})\lesssim \frac{1}{\lambda}\|f\|_{L_{\sigma}^1(D)}.
$$
\end{proof}

\begin{proof}[Proof of Proposition \ref{BergmanAuxiliaryWeakType}]
Since $\mathcal{T}_2$ has a bounded kernel and $\sigma(D)<\infty$ (using the $B_1$ condition), it is immediate that $\mathcal{T}_2$ is bounded on $L_{\sigma}^1(D)$, and hence from $L_{\sigma}^1(D)$ to $L_{\sigma}^{1,\infty}(D)$. It is thus sufficient to prove the estimate for $\mathcal{T}_1$. 

As in \cite{WW2020}, we define a positive comparison operator $\Gamma$ by 
$$
    \Gamma f(z) :=\int_{D}\frac{f(w)}{|g(z,w)|^{n+1}} \, dV(w).
$$
It can easily be shown that
$$
    |\mathcal{T}_1f(z)| \lesssim \Gamma|f|(z),
$$
so it suffices to prove that $\Gamma$ maps $L_{\sigma}^1(D)$ to $L_{\sigma}^{1,\infty}(D)$.

Let $f$ be a nonnegative and continuous function on $D$ and let $\lambda >0$. We will show that $$
    \sigma(\{\Gamma f>\lambda\}) \lesssim \frac{1}{\lambda}\|f\|_{L_{\sigma}^1(D)}.
$$
A density argument and doubling the implied constant in the display above yields the result for general $f \in L_{\sigma}^1(D)$.

Apply Lemma \ref{CZD} to write 
$$
    f\approx f_1+f_2\approx f_1+\sum_{B\in\mathcal{F}}f_{2,B},
$$
where the properties and notations from the lemma hold. Then 
\begin{align*}
    \sigma(\{\Gamma f>\lambda\}) &\leq \sigma\left(\left\{\Gamma f_1>\frac{\lambda}{C}\right\}\right)+\sigma\left(\left\{\Gamma f_2>\frac{\lambda}{C}\right\}\right)\\
    &\leq \sigma\left(\left\{\Gamma f_1>\frac{\lambda}{C}\right\}\right)+\sigma\left(\bigcup_{B\in\mathcal{F}}RB\right)+\sigma\left(\left\{z \in D\setminus\bigcup_{B\in\mathcal{F}}RB: \Gamma f_2(z)>\frac{\lambda}{C}\right\}\right)
\end{align*}
for some $C>0$ and where $R>1$ will be fixed later. Therefore it is enough to bound 
\begin{align*}
    &\text{I}:=\sigma(\{\Gamma f_1 >\lambda\}),\\
    &\text{II}:=\sigma\left(\bigcup_{B\in\mathcal{F}}RB\right),\quad\text{and}\\
    &\text{III}:=\sigma\left(\left\{z \in D\setminus\bigcup_{B\in\mathcal{F}}RB: \Gamma f_2(z)>\lambda\right\}\right)
\end{align*}
by constants multiplied by $\frac{1}{\lambda}\|f\|_{L_{\sigma}^1(D)}$.

To address $\text{I}$, we first claim that there exists $k>0$ such that for all integrable and nonnegative $u$, we have 
$$
    \Gamma u(z) \lesssim \Gamma(R_{k'}u)(z),
$$ 
where $k'=\frac{ck}{1-ck}$. Indeed, by Lemma \ref{gisconstant}, we have $|g(z,w)|\approx |g(z,w')|$ for all $z \in D$ and $w'\in B_k(w)$. Using the above and Lemma \ref{RkRk'}, we deduce
\begin{align*}
    \Gamma u(z)&=\int_D \frac{1}{|g(z,w)|^{n+1}}u(w)\,dV(w)\\
    &\approx \int_D \left(\frac{1}{V(B_k(w))}\int_{B_k(w)}\frac{1}{|g(z,w')|^{n+1}}\,dV(w')\right)u(w)\,dV(w)\\
    &\lesssim \int_D \frac{1}{|g(z,w)|^{n+1}}\left(\frac{1}{V(B_{k'}(w))}\int_{B_{k'}(w)}u(w')\,dV(w')\right)dV(w)\\
    &=\Gamma(R_{k'}u)(z).
\end{align*}
Therefore, using Chebyshev's inequality, the above claim, the $L_{\sigma}^2(D)$ bound of $\Gamma$ (see \cite{WW2020}), property (1) of Lemma \ref{CZD}, Lemma \ref{RkRk'}, and the $B_1$ condition of $\sigma$ we have
\begin{align*}
    \text{I}& \lesssim \frac{1}{\lambda^2}\int_D(\Gamma f_1)^2\sigma\,dV\\
    &\lesssim \frac{1}{\lambda^2}\int_D(\Gamma(R_{k'}f_1))^2\sigma\,dV\\
    &\lesssim \frac{1}{\lambda^2}\int_D (R_{k'}f_1)^2\sigma\,dV\\
    &\lesssim \frac{1}{\lambda}\int_D(R_{k'}f_1)\sigma\,dV\\
    &\lesssim \frac{1}{\lambda}\int_D f_1(R_{k''}\sigma)\,dV\\
    &\lesssim\frac{1}{\lambda}\int_Df_1\sigma\,dV\\
    &\leq \frac{1}{\lambda}\|f\|_{L_{\sigma}^1(D)}.
\end{align*}

The control of $\text{II}$ follows from the doubling property of $\sigma$ and property (3) of Lemma \ref{CZD}:
$$
    \text{II}\leq \sum_{B\in\mathcal{F}}\sigma(RB)\lesssim \sum_{B\in\mathcal{F}}\sigma(B)\lesssim \frac{1}{\lambda}\|f\|_{L_{\sigma}^1(D)}.
$$

For $\text{III}$, we claim that if $R>1$ is sufficiently large and $u$ is supported on a quasi-ball $B$, then
$$
    \Gamma u(z) \lesssim \Gamma\left(\langle u \rangle_B\chi_B\right)(z)
$$
for all $z \in D \setminus RB$. Indeed, as stated in the proof of Lemma \ref{gisconstant}, we have $|g(z,w)|\approx |g(z,w')|$ whenever $d(w,w')\leq Cd(z,w)$. For $z \in D\setminus RB$ and $w,w' \in B$, we use the triangle inequality to obtain $d(w,w')< 2cr(B)$ and $\frac{R-c}{c}r(B)< d(z,w)$. Thus, if $R$ is chosen large enough so that $2c\leq C\frac{R-c}{c}$, we have $d(w,w')\leq Cd(z,w)$, and hence $|g(z,w)|\approx |g(z,w')|$. The claim follows via using Fubini's theorem
\begin{align*}
    \Gamma u(z)&= \int_B \frac{1}{|g(z,w)|^{n+1}}u(w)\,dV(w)\\
    &\approx \int_B \left(\frac{1}{V(B)}\int_B\frac{1}{|g(z,w')|^{n+1}}\,dV(w')\right)u(w)\,dV(w)\\
    &=\int_B\frac{1}{|g(z,w')|^{n+1}}\left(\frac{1}{V(B)}\int_Bu(w)\,dV(w)\right)\,dV(w')\\
    &=\Gamma\left(\langle u \rangle_B \chi_B\right)(z).
\end{align*}

Using the above claim, we have
$$
    \Gamma f_2(z)\approx\sum_{B \in \mathcal{F}} \Gamma f_{2,B}(z) \lesssim \sum_{B\in\mathcal{F}} \Gamma\left(\langle f_{2,B}\rangle_B\chi_B\right)(z)=\Gamma\tilde{f}_2(z)
$$
for $z \in D\setminus \bigcup_{B\in\mathcal{F}}RB$, where $\tilde{f}_2:=\sum_{B\in\mathcal{F}}\langle f_{2,B}\rangle_B\chi_B$.
Therefore, to control \text{III}, it suffices to prove 
$$
    \sigma(\{ \Gamma\tilde{f}_2>\lambda\}) \lesssim \frac{1}{\lambda}\|f\|_{L_{\sigma}^1(D)}.
$$
To accomplish this, apply Chebyshev's inequality, the bound of $\Gamma$ on $L_{\sigma}^2(D)$, property (2) of Lemma \ref{CZD}, the $B_1$ condition of $\sigma$, and the almost disjointness of the quasi-balls in $\mathcal{F}$
\begin{align*}
    \sigma(\{\Gamma\tilde{f}_2>\lambda\})&\lesssim \frac{1}{\lambda^2}\int_{D}(\Gamma\tilde{f}_2)^2\sigma\,dV\\
    &\lesssim \frac{1}{\lambda^2}\int_{D}\tilde{f}_2^2\sigma\,dV\\
    &\lesssim \frac{1}{\lambda}\sum_{B\in\mathcal{F}}\int_B\langle f\rangle_B\sigma\,dV\\
    &\lesssim \frac{1}{\lambda}\sum_{B\in\mathcal{F}}\int_B f\sigma\,dV\\
    &\lesssim \frac{1}{\lambda}\|f\|_{L_{\sigma}^1(D)}.
\end{align*}



\end{proof}


\section{The Cauchy-Szeg\H o Projection}\label{CauchySzegoSection}
Throughout this section, we assume that the domain $D$ has class $C^3$ boundary. This means the same thing as in Section \ref{BergmanSection} except the defining function $\rho$ is only assumed to be of class $C^3.$ We construct an auxiliary operator, $\mathcal{C}$, that produces and reproduces boundary values of holomorphic functions. This construction proceeds in a similar way to the construction in Section \ref{BergmanSection}, and we will reuse certain notations to refer to objects playing analogous roles. For more details concerning the construction of this operator, we refer the reader to \cite{LS2017}.

Let $P_w(z)$ denote the Levi polynomial, but this time at $w \in bD$. In this case, we set 
$$
    g(z,w)= -\chi P_w(z)+(1-\chi)|z-w|^2,
$$
where $\chi=\chi(z,w)$ is an appropriately chosen $C^\infty$ cutoff function with  $\chi \equiv 1$ when $|z-w|<\delta/2$ and $\chi \equiv 0$ when $|z-w| >\delta$ for some $\delta>0$ so that
$$
    \text{Re} \,g(z,w) \gtrsim -\rho(z)+|z-w|^2.
$$
The $(1,0)$ form in $w$, $G(z,w)$, is defined exactly the same way as in the Bergman case. This time we have  
$$
    \langle G(z,w), w-z\rangle = g(z,w).
$$
As before, we set 
$$
    \eta(z,w):= \frac{G(z,w)}{g(z,w)}
$$
and define similarly for $z \in D$
$$
    \mathbf{C}_1 f(z):=\frac{1}{(2 \pi \mathrm{i})^n}\int_{bD} j^*\left((\overline{\partial}_w\eta)^{n-1} \wedge \eta\right) f(w)
$$
where $j: bD \rightarrow \mathbb{C}^n$ denotes the inclusion map. Notice now the integration takes place over the boundary rather than the interior of the domain.  

We have the following lemma, which grants that $\mathbf{C}_1$ reproduces holomorphic functions from their boundary values. 
\begin{lemma}\label{reproduce 2}
If $F$ is holomorphic on $D$, continuous on $\overline{D}$, and $f=F|_{bD}$, then for all $z \in D$,
$$
    \mathbf{C}_1f(z)=F(z).
$$
\end{lemma}

The problem again is that $\mathbf{C}_1$ does not generally produce holomorphic functions. This is corrected with the following lemma.
\begin{lemma}\label{produce and reproduce}
There is a continuous $(n,n-1)$ form $C_2(z,w)$ in $w$ that depends smoothly on the parameter $z \in \overline{D}$ so that if we define
$$
    \mathbf{C}_2f(z):= \int_{bD}C_2(z,w) f(w)
$$ 
and 
$$
    \mathbf{C}:=\mathbf{C}_1+\mathbf{C}_2,
$$ 
then the operator $\mathbf{C}$ satisfies the following:
\begin{enumerate}
    \item If $f \in L^1(bD)$, then $\mathbf{C}f$ is holomorphic on $D$. 
    \item If $F$ is holomorphic on $D$, continuous on $\overline{D}$, and $f=F|_{bD}$, then $\mathbf{C}f(z)=F(z)$ for all $z \in D$.
\end{enumerate}
\end{lemma}

We define the operator $\mathcal{C}f:=\mathbf{C}f|_{bD}$ for a class of functions satisfying a certain type of H\"{o}lder continuity and refer to \cite{LS2017} for the details. It was proved in \cite{LS2017} that the operator $\mathcal{C}$ extends boundedly on $L^p(bD)$ for all $1<p<\infty.$ The Kerzman-Stein equation now takes the following form:
$$
    \mathcal{S}(I-(\mathcal{C}^*-\mathcal{C}))=\mathcal{C}.
$$

To prove that $\mathcal{S}$ is weak-type $(1,1)$ with respect to the $A_1$ weight $\sigma$, we proceed in two steps as before. In particular, we have the following two propositions:

\begin{prop}\label{SzegoAuxiliaryInvertible}
If $\sigma$ is an  $A_1$ weight, then the operator $I-(\mathcal{C}^*-\mathcal{C})$ is invertible on $L^1_{\sigma}(bD)$.
\end{prop}

\begin{prop}\label{SzegoAuxiliaryWeakType}
If $\sigma$ is an $A_1$ weight, then $\mathcal{C}$ maps $L^1_\sigma(bD)$ to $L^{1,\infty}_{\sigma}(bD)$ boundedly.
\end{prop}

We can now prove Theorem \ref{CauchySzegoWeakType}.
\begin{proof}[Proof of Theorem \ref{CauchySzegoWeakType}]
This follows directly from Proposition \ref{SzegoAuxiliaryInvertible} and Proposition \ref{SzegoAuxiliaryWeakType}.
\end{proof}

The proof of Proposition \ref{SzegoAuxiliaryInvertible} proceeds as in the Bergman case. We again appeal to Lemma \ref{compact operator} and Lemma \ref{relatively compact} to prove a compactness result. In this case, the underlying space is $X=bD$ and the finite Borel measure is $\sigma \,dS$. We can define the appropriate quasi-metric:
$$
    d(z,w):=|g(z,w)|^{\frac{1}{2}}.
$$
It was proved in \cite{LS2017} that $d$ is indeed a quasi-metric and that $(D,d,S)$ is a space of homogeneous type. Additionally, we have $S(B(z,r)) \approx r^{2n}.$

\begin{lemma} \label{CompactnessSzego}
If $\sigma$ is an $A_1$ weight, then the operator $\mathcal{C}^*-\mathcal{C}$ is compact on $L^1_{\sigma}(bD)$.
\end{lemma} 

\begin{proof}
Let $k(z,w)$ denote the kernel of the operator $\mathcal C ^*-\mathcal C.$ We consider the family of functions $k_w(z)=k(z,w)$ for $w \in bD.$ By Lemma \ref{compact operator}, it suffices to show that $\{k_w: w \in bD\}$ is relatively compact in $L^1_{\sigma}(bD)$, which we can do by verifying the criteria of Lemma \ref{relatively compact}. The infimum condition can be verified as before. This set is clearly bounded in $L^1_{\sigma}(bD)$; this follows by observing as in \cite{WW2020} that $$\int_{bD}|k(z,w)| \sigma(z) \mathop{d S(z)} \lesssim \sigma(w).$$ In particular, this is deduced from the bound $|k(z,w)| \lesssim d(z,w)^{-2n+1}$ (which relies on the domain having boundary of class $C^3$) and a dyadic integration argument similar to the one presented in Lemma \ref{BergmanLCompactness}. Similarly, we obtain that 
$$\int_{B(w,\delta)}|k(z,w)| \sigma(z) \mathop{d S(z)} \lesssim \delta \sigma(w).$$ Notice that this bound does not involve a $d(z,bD)$ term which highlights a key difference from the case of the Bergman projection. 

The second conclusion mirrors very closely the argument in the Section \ref{BergmanSection}, so we only sketch the ideas. Namely, for a fixed function $k_w$, we excise a small ball about $w$ and integrate the function $|k_w(z)-\langle k_w \rangle_{B(z,r), \sigma dS}|$ over this ball and its complement. The integral on the complement of the ball can be controlled by uniform continuity, since $k(z,w)$ is continuous off the boundary diagonal. The integral over the ball is controlled via the triangle inequality and splitting into regions as in the proof of Lemma \ref{BergmanLCompactness}. It should be noted that it is not necessary to split into subcases based on the distance of points $z$ and $w$ to $bD$ because all the integration occurs on the boundary and $A_1$ weights satisfy a true doubling property.
\end{proof}

The following lemma follows the exact same argument as Lemma \ref{Bergmanspectrum}.

\begin{lemma}\label{SzegoSpectrum}
If $\sigma \in A_1$, the number $1$ is not an eigenvalue of $\mathcal{C}^*-\mathcal{C}$ considered as an operator on $L^1_{\sigma}(bD)$.
\end{lemma}

Therefore, we can prove Proposition \ref{SzegoAuxiliaryInvertible}:
\begin{proof}[Proof of Proposition \ref{SzegoAuxiliaryInvertible}]
This proposition follows from Lemma \ref{CompactnessSzego}, Lemma \ref{SzegoSpectrum}, and the spectral theorem for compact operators on a Banach space.
\end{proof}

To complete the proof of Theorem \ref{CauchySzegoWeakType}, it remains to prove Proposition \ref{SzegoAuxiliaryWeakType}.

\begin{proof}[Proof of Proposition \ref{SzegoAuxiliaryWeakType}]
As in \cite{LS2017}, write $\mathcal{C}=\mathcal{C}^\sharp+\mathcal{R}$. It is proven in \cite{LS2017} that $\mathcal{C}^\sharp$ is a Calder\'{o}n-Zygmund operator with respect to the quasi-metric $d$. Thus, by standard theory, $\mathcal{C}^\sharp$ maps $L_{\sigma}^1(bD)$ to $L_{\sigma}^{1,\infty}(bD)$ boundedly for $\sigma \in A_1$. 

On the other hand, the operator $\mathcal{R}$ has a kernel $r(z,w)$ that satisfies
$$
    \int_{bD}|r(z,w)| \sigma(w) \mathop{d S(w)} \lesssim \sigma(z)
$$ 
and
$$
    \int_{bD}|r(z,w)| \sigma(z) \mathop{d S(z)} \lesssim \sigma(w)
$$
for $\sigma \in A_1$ (see \cite{WW2020}). A simple argument using Fubini's theorem shows that $\mathcal{R}$ is bounded on $L^1_{\sigma}(bD).$ This completes the proof. 
\end{proof} 


\section{Kolmogorov and Zygmund Inequalities}\label{BonusEstimates}

We first prove the general fact that the weak-type $(1,1)$ estimate implies the Kolmogorov inequality on finite measure spaces.

\begin{thm}\label{KolmogorovGeneral}
Let $T$ be a linear operator and $(X,\mu)$ a finite measure space. If $T$ maps $L^1(X,\mu)$ to $L^{1,\infty}(X,\mu)$ boundedly and $0<p<1$, then $T$ extends boundedly from $L^1(X,\mu)$ to $L^p(X,\mu)$. 
\end{thm}

\begin{proof}
Using the distribution function and the weak-type $(1,1)$ assumption, we have for any $t>0$:
\begin{align*}
 \|Tf&\|_{L^p(X,\mu)}^p  =  \int_{0}^{\infty} p \lambda^{p-1} \mu(\{x\in X:|Tf(x)|>\lambda\}) \mathop{d \lambda}\\
 & = \int_{0}^{t} p \lambda^{p-1} \mu(\{x\in X:|Tf(x)|>\lambda\}) \mathop{d \lambda} + \int_{t}^{\infty} p \lambda^{p-1} \mu(\{x\in X:|Tf(x)|>\lambda\}) \mathop{d \lambda} \\
 & \leq t^p \mu(X)+ \frac{p}{1-p} t^{p-1} \|f\|_{L^1(X,\mu)}.
\end{align*}
Taking $t= \|f\|_{L^1(X,\mu)}$ completes the proof.
\end{proof}

\begin{proof}[Proof of Corollary \ref{KolmogorovBergman}]
This follows immediately from Theorem \ref{BergmanWeakType} and Theorem \ref{KolmogorovGeneral}.
\end{proof}

\begin{proof}[Proof of Corollary \ref{KolmogorovSzego}]
This follows immediately from Theorem \ref{CauchySzegoWeakType} and Theorem \ref{KolmogorovGeneral}.
\end{proof}

Before proving our Zygmund inequalities, we first define the space $L\log^+ L$, which falls within the scope of Orlicz spaces. We call a function $\Phi:[0,\infty]\rightarrow [0,\infty]$ a Young function if $\Phi$ is continuous, convex, increasing, and satisfies $\Phi(0)=0$. Given a measure space $(X,\mu)$ and a Young function $\Phi$, the associated Orlicz space, $L^{\Phi}(X,\mu)$, is the linear hull of all measurable functions on $X$ satisfying
$$
    \int_X\Phi(|f|)\,d\mu < \infty
$$
equipped with the following Luxemburg norm:
$$
    \|f\|_{L^{\Phi}(X,\mu)}:=\inf\left\{\lambda >0 : \int_X \Phi\left(\frac{|f|}{\lambda}\right)\,d\mu \leq 1\right\}.
$$

The Zygmund space $L\log^+ L(X,\mu)$ is defined to be the Orlicz space $L^{\Psi}(X,\mu)$ associated with the Young function $\Psi(t)=t\log^+ t$, where $\log^+(t):=\max\{\log(t),0\}$. We use the notation $(L\log^+ L)_{\sigma}(D)$ to represent $L \log^+ L(D,\sigma\,dV)$ for a domain $D\subseteq \mathbb{C}^n$ and a weight $\sigma$ on $D$ and we similarly write $(L\log^+ L)_{\sigma}(bD)$ for $L\log^+ L (bD,\sigma \,dS)$ with $\sigma$ a weight on $bD$. We refer to \cites{KR1961,RR1991} for thorough treatments of Orlicz spaces.

We next prove that the weak-type $(1,1)$ and $L^2$ bounds imply the Zygmund inequality on general finite measure spaces. 
\begin{thm}\label{ZygmundGeneral}
Let $T$ be a linear operator and $(X,\mu)$ a finite measure space. If $T$ is bounded on $L^2(X,\mu)$ and maps $L^1(X,\mu)$ to $L^{1,\infty}(X,\mu)$ boundedly, then $T$ extends boundedly from $L \log^+ L(X,\mu)$ to $L^1(X,\mu)$.
\end{thm}
\begin{proof}
Let $f \in L \log^+ L(X,\mu)$ be given and normalized to assume $\|f\|_{L\log^+L(X,\mu)}=1$. 
Observe that $L^1(X,\mu)$ is the Orlicz space $L^{\Phi}(X,\mu)$ with Young function $\Phi(t)=t$. Define $\Phi_1$ by 
$$
\Phi_1(t)=
    \begin{cases}
        0 & \text{if}\quad 0\leq t<2\\
        t-2 & \text{if}\quad 2\leq t \leq \infty
    \end{cases}
$$
and notice that $\Phi$ and $\Phi_1$ are equivalent Young functions in the sense that 
$$
    \Phi_1(t)\leq\Phi(t)\leq\Phi_1(2t)
$$
for all $t\ge 2$. 
Therefore by \cite[Theorem 13.2 and Theorem 13.3]{KR1961}, it suffices to prove 
$$
    \|Tf\|_{L^{\Phi_1}(X,\mu)}\lesssim 1.
$$

For a fixed $\lambda>0$, write $f=f_0+f_{\infty}$, where $f_0:= f\chi_{\{|f|\leq\lambda\}}$ and $f_{\infty}:=f\chi_{\{|f|>\lambda\}}$. Using the assumed bounds of $T$ and the distribution function, we have 
\begin{align*}
    \mu(\{|Tf|>2\lambda\}) &\leq \mu(\{|Tf_0|>\lambda\}) + \mu(\{|Tf_{\infty}|>\lambda\})\\
    &\leq \frac{1}{\lambda^2}\|f_0\|_{L^2(X,\mu)}^2+\frac{1}{\lambda}\|f_{\infty}\|_{L^1(X,\mu)}\\
    &\approx \frac{1}{\lambda^2}\int_{0}^{\lambda}s\mu(\{|f|>s\})\,ds + \frac{1}{\lambda}\int_{\lambda}^{\infty} \mu(\{|f|>s\})\,ds.
\end{align*}
Use the distribution function, a change of variables, the above estimate, and Fubini's Theorem, direct estimates, and the normalization $\|f\|_{L\log^+ L(X,\mu)}=1$ to deduce
\begin{align*}
    \int_X \Phi_1(|Tf|)\,d\mu &= \int_{2}^{\infty} \mu(\{|Tf|>\lambda\})\,d\lambda \approx \int_{1}^{\infty}\mu(\{|Tf|>2\lambda\})\,d\lambda\\
    &\leq \int_{1}^{\infty}\frac{1}{\lambda^2}\int_{0}^{\lambda}s\mu(\{|f|>s\})\,dsd\lambda + \int_{1}^{\infty}\frac{1}{\lambda}\int_{\lambda}^{\infty} \mu(\{|f|>s\})\,dsd\lambda\\
    &= \int_0^1 s\mu(\{|f|>s\}) \int_1^{\infty}\frac{1}{\lambda^2}\,d\lambda ds + \int_1^{\infty}s\mu(\{|f|>s\})\int_s^{\infty}\frac{1}{\lambda^2}\,d\lambda ds \\
    &\quad\quad + \int_{1}^{\infty}\mu(\{|f|>s\})\int_{1}^s\frac{1}{\lambda}\,d\lambda ds\\
    &=\int_0^1s\mu(\{|f|>s\})\,ds+\int_1^{\infty}(1+\log s)\mu(\{|f|>s\})\,ds\\
    &\leq \mu(X) + \int_{X} \Psi(|f|)\,d\mu\\
    &\lesssim 1,
\end{align*}
where $\Psi(t)=t\log^+(t)$. Thus $\|Tf\|_{L^{\Phi_1}(X,\mu)}\lesssim 1$ as desired.
\end{proof}

\begin{proof}[Proof of Corollary \ref{ZygmundBergman}]
This follows immediately from Theorem \ref{BergmanWeakType} and Theorem \ref{ZygmundGeneral}.
\end{proof}

\begin{proof}[Proof of Corollary \ref{ZygmundSzego}]
This follows immediately from Theorem \ref{CauchySzegoWeakType} and Theorem \ref{ZygmundGeneral}.
\end{proof}


\begin{bibdiv}
\begin{biblist}
\bib{B198182}{article}{
title={In\'egalit\'e \`a poids pour le projecteur de Bergman dans la boule unit\'e de $\mathbb{C}^n$},
author={D. Bekoll\'e},
journal={Studia Math.},
volume={71},
date={1981/82},
number={3},
pages={305--323}
}

\bib{CW1971}{book}{
title={Analyse harmonique non-commutative sur certains espaces homog\`enes. (French) \'Etude de certaines int\'egrales singuli\`eres},
author={R. R. Coifman},
author={G. Weiss},
publisher={Springer-Verlag},
address={Berlin-New York},
series={Lecture Notes in Mathematics},
volume={242},
date={1971}
}

\bib{DHZZ2001}{article}{
title={Bergman projection and Bergman spaces},
author={Y. Deng},
author={L. Huang},
author={T. Zhao},
author={D. Zheng},
journal={J. Operator Theory},
volume={46},
date={2001},
number={1},
pages={3--24}
}

\bib{E1995}{article}{
title={Criteria for integral operators in $L^{\infty}$ and $L^1$ spaces},
author={S. P. Eveson},
journal={Proc. Amer. Math. Soc.},
volume={123},
number={12},
date={1995},
pages={3709--3716}
}

\bib{FR1975}{article}{
title={Projections on spaces of holomorphic functions in balls},
author={F. Forelli},
author={W. Rudin},
journal={Indiana Univ. Math. J.},
volume={24},
date={1974/75},
pages={593--602}
}

\bib{HWW2020}{article}{
title={A Bekoll\'e-Bonami Class of Weights for Certain Pseudoconvex Domains},
author={Z. Huo},
author={N. A. Wagner},
author={B. D. Wick},
date={2020},
journal={Arxiv e-prints: 2001.08302}
}

\bib{HMW1973}{article}{
title={Weighted norm inequalities for the conjugate function and Hilbert transform},
author={R. Hunt},
author={B. Muckenhoupt},
author={R. Wheeden},
journal={Trans. Amer. Math. Soc.},
volume={176},
date={1973},
pages={227--251}
}

\bib{IK2009}{article}{
title={Compactness of embeddings of Sobolev type on metric measure spaces},
author={I. A. Ivanishko},
author={V. G. Krotov},
journal={Math. Notes},
volume={86},
date={2009},
number={5-6},
pages={775--788}
}


\bib{K1999}{article}{
title={On compactness of embedding for Sobolev spaces defined on metric spaces},
author={A Kalamajska},
journal={Ann. Acad. Sci. Fenn. Math.},
volume={24},
number={1},
pages={123--132},
date={1999}
}

\bib{KR1961}{book}{
title={Convex functions and Orlicz spaces},
author={M. A. Krasnosel'ski\u i},
author={J. A. Ruticki\u i},
publisher={P. Noordhoff Ltd.},
address={Groningen},
date={1961}
}


\bib{LS2012}{article}{
title={The Bergman projection in $L^p$},
author={L. Lanzani},
author={E. M. Stein},
journal={Illinois J. Math.},
volume={56},
date={2012},
number={1},
pages={127 -- 154}
}

\bib{LS2017}{article}{
title={The Cauchy-Szeg\H o projection for domains in $\mathbb{C}^n$ with minimal smoothness},
author={L. Lanzani},
author={E. M. Stein},
journal={Duke Math. J.},
volume={166},
date={2017},
number={1},
pages={125--176}
}

\bib{M2003}{article}{
title={Subelliptic estimates and scaling in the $\bar \partial$-Neumann problem},
author={J. D. McNeal},
journal={Contemp. Math.},
volume={332},
date={2003},
pages={197-217}
}

\bib{M1994}{article}{
title={The Bergman projection as a singular integral operator},
author={J. D. McNeal},
journal={J. Geom. Anal.},
volume={4},
date={1994},
number={1},
pages={91--103}
}

\bib{PS1977}{article}{
title={Estimates for the Bergman and Szeg\H o projections on the pseudo-convex domains},
author={D. E. Phong},
author={E. M. Stein},
journal={Duke Math. J.},
volume={44},
date={1977},
number={3},
pages={695--704},
}

\bib{R1986}{book}{
title={Holomorphic functions and integral representations in several complex variables},
author={M. Range},
publisher={Springer},
address={Berlin},
date={1986},
}

\bib{RR1991}{book}{
title={Theory of Orlicz space},
author={M. M. Rao},
author={Z. D. Ren},
series={Monographs and Textbooks in Pure and Applied Mathematics},
volume={146},
publisher={Marcel Dekker Inc.},
address={New York},
date={1991}
}

\bib{WW2020}{article}{
title={Weighted $L^p$ Estimates for the Bergman and Szeg\H{o} Projections on Strongly Pseudoconvex Domains with Near Minimal Smoothness},
author={N. A. Wagner},
author={B. D. Wick},
journal={Arxiv e-prints: 2004.10248},
date={2020}
}

\bib{ZJ1964}{article}{
title={The general form of a linear functional in $H_p'$ (Russian)},
author={V. P. Zaharjuta},
author={V. I. Judovi\v c},
journal={Uspehi Mat. Nauk},
volume={19},
date={1964}, 
number={2 (116)}, 
pages={139--142}
}

\bib{Zhu}{book}{
title={Spaces of holomorphic functions on the unit ball},
author={K. Zhu},
edition={Third edition},
series={Graduate Texts in Mathematics},
publisher={Springer},
address={New York},
date={2005},
volume={226}
}
\end{biblist}
\end{bibdiv}
\end{document}